\documentclass[10pt]{article}
\usepackage{amsmath, amssymb, amsfonts, amsthm}
\usepackage{enumitem}
\usepackage{multicol}
\allowdisplaybreaks
\def \R {{\mathbb{R}}}
\def \Q {{\mathbb{Q}}}
\def \N {{\mathbb{N}}}
\def \Z {{\mathbb{Z}}}
\def \C {{\mathbb{C}}}

\newcommand{\sgn}{\operatorname{sgn}}

\newtheorem*{theorem*}{Theorem}
\newtheorem{theorem}{Theorem}
\newtheorem{cor}[theorem]{Corollary}

\newtheorem{lemma}[theorem]{Lemma}

\newtheorem{rem}[theorem]{Remark}

\title{An irreducibility criterion for polynomials over integers}
\author {Biswajit Koley,  A.Satyanarayana Reddy \\
Department of 
Mathematics, Shiv Nadar 
University, India-201314\\ (e-mail: 
bk140@snu.edu.in, satyanarayana.reddy@snu.edu.in).
  }
\date{}
\begin{document}
\maketitle
\begin{abstract}
In this article, we consider the polynomials of the form $f(x)=a_0+a_1x+a_2x^2+\cdots+a_nx^n\in \Z[x],$
where $|a_0|=|a_1|+\dots+|a_n|$ and $|a_0|$ is a prime.
We show that these polynomials have a cyclotomic factor whenever reducible. As a consequence, we give a  simple procedure for checking the irreducibility of trinomials of this form and separability criterion for certain quadrinomials.
\end{abstract}
{\bf{AMS(2010)}}: 11R09, 12D05, 12D10.\\

\section{Introduction}\label{section:intro}
 Polynomials of the form $ax^n\pm b$ have cyclotomic factors if $a=b.$ W. Ljunggren~\cite{lju} extended the work of Selmer~\cite{sel} and showed that a polynomial of form  $x^n\pm x^m \pm 1$ or  $x^n\pm x^m\pm x^r\pm 1$ is either irreducible or divisible by a cyclotomic polynomial.  The following result 
 provides a criterion for the irreducibility of polynomials with integer  coefficients.

\begin{theorem}[L. Panitopol, D. Stef\"anescu~\cite{LP}]\label{thm:LP}
Let $f(x)=a_0+a_1x+a_2x^2+\cdots+a_nx^n\in \Z[x]$ be a polynomial such that $|a_0|>|a_1|+\cdots+|a_n|$. If $|a_0|$ is a prime or $\sqrt{|a_0|}-\sqrt{|a_n|}<1$ then $f(x)$ is irreducible in $\Z[x]$. 
\end{theorem}

In this paper 
we show that, if a polynomial  of the from $f(x)=a_0+a_1x+a_2x^2+\cdots+a_nx^n\in \Z[x],$
where $|a_0|=|a_1|+\dots+|a_n|$ and $|a_0|$ is a prime is reducible, then it is divisible by a cyclotomic polynomial and  has exactly one non-reciprocal factor. The condition  that $|a_0|$ is a prime number is necessary. If $|a_0|$ is not a  prime number, then $f(x)$ can have any number of   non-reciprocal factors. For example,  
\begin{equation*}
x^4+3x^2+4=(x^2-x+2)(x^2+x+2).
\end{equation*}
 More generally, every member of 
\begin{equation*}
x^{4n}-(a^2-1)x^{2n}-a^2=(x^n+a)(x^n-a)(x^n+1),
\end{equation*}
where $a\ge 2$ is an integer and $n\ge 1$, has at least two non-reciprocal factors. By use of Capelli's theorem \cite{capelli} it can be shown that arbitrary number of non-reciprocal factors can appear for suitable choice of $a$.  

But if the absolute value of the product of all the roots of a non-trivial factor of $f(x)$ belongs to $(0,1]$, then one can show that $f(x)$ has a cyclotomic factor even if $|a_0|$ is not a prime number.
\begin{theorem}\label{gen}
Let $f(x)=a_0+a_1x+a_2x^2+\cdots+a_nx^n\in \Z[x],$ where  $|a_0|=|a_1|+\dots+|a_n|$. If $g(x)=b_0+b_1x+\dots+b_mx^m$ is a factor of $f(x)$ with $0<|b_0|\le |b_m|,$ then $g(x)$ is a product of cyclotomic polynomials.
\end{theorem}

If $|a_0|$ is a prime number, then not only such a polynomial $g(x)$  defined in  Theorem~\ref{gen} exists, but also $f(x)$ has exactly one non-reciprocal factor. Further all the roots of $f(x)$ are simple zeros. Before proceeding, by removing zero coefficients if any, we can assume that $f(x)=a_{n_r}x^{n_r}+\cdots+a_{n_1}x^{n_1}+a_0$ where $a_0a_{n_1}\cdots a_{n_r}\ne 0.$ 

\begin{theorem}\label{thm:main}
Let $f(x)=a_{n_r}x^{n_r}+\cdots+a_{n_1}x^{n_1}+a_0,$ where $|a_{n_1}|+\cdots+|a_{n_r}|=|a_0|$ and $|a_0|$ be a prime number. If $f$ is reducible, then 
$f(x)=f_c(x)f_n(x)$, where $f_n(x)$ is the irreducible non-reciprocal factor of $f(x)$ and $f_c(x)$ is the product of all cyclotomic factors of $f(x),$ in particular $f_c(x)=\gcd(x^{n_i}+\sgn(a_0a_{n_i})),$ where $1\le i\le r$ and $\sgn(x)$ denotes the sign of $x\in \R$.
\end{theorem}

Above theorem generates several other irreducibility criteria for different family of polynomials. Before stating them, we denote the {\em largest even part} of $n$ as $e(n)$, {\it i.e.,} if $n=2^{\alpha}n_1$ with $n_1$ odd, then $e(n)=2^{\alpha}$.

\begin{cor}\label{cor:pos}
Suppose $f(x)=a_{n_r}x^{n_r}+\cdots+a_{n_1}x^{n_1}+a_0\in \Z_+[x]$ be a polynomial with $a_{n_1}+\dots+a_{n_r}=a_0$ and $a_0$ is  a prime number.  Then $f(x)$ is irreducible if and only if there exist distinct $i,j$ such that $e(n_i)\ne e(n_j).$  
\end{cor}

Corollary~\ref{cor:pos} provides an extensive class of irreducible polynomials. The simplest among them is: If $p$ is a number and $n_1,n_2,\ldots, n_p$ are positive integers with $(n_1,n_2, \ldots, n_p)=1$, then $x^{n_1}+x^{n_2}+\cdots+x^{n_p}+p$ is irreducible. 

On the other hand, if we have two consecutive exponents then the irreducibility criterion for such polynomials become easier. 
\begin{cor}\label{cor:main}
Suppose $f(x)=a_{n_r}x^{n_r}+a_{n_{r-1}}x^{n_{r-1}}+\cdots+a_0$ be a polynomial with $n_r=n_{r-1}+1$. Let $|a_{n_1}|+\cdots+|a_{n_r}|=|a_0|$ and $|a_0|$ is a prime number. Then $f(x)$ is reducible if and only if $f(1)=0$ or $f(-1)=0$. 
\end{cor}
 The proof of this Corollary directly follows from Theorem \ref{thm:main} along with the fact that $(x^n\pm 1, x^{n-1}\pm 1)$ is either  $1$ or $x\pm 1$. The condition $n_r=n_{r-1}+1$ can be replaced by $n_j=n_{j-1}+1$ for $2\le j\le r$.

In Section~\ref{sec:app} we present an explicit irreducibility criterion for trinomials of the form $ax^n+b\epsilon_1x^m+p\epsilon_2$ with $a+b=p,$ where $a,b,p\in \Z_+,$ $p$ is  a prime number and $\epsilon_i\in \{-1,1\}$.  In the end, we  determine the separability of quadrinomials of  form $x^n+\epsilon_1x^m+\epsilon_2x^r+\epsilon_3, \epsilon_i\in \{-1, 1\}$ over $\Q$, the set of rational numbers.

\section{Proofs}\label{sec:main}
In this section, we give  proof for Theorem~\ref{thm:main} and Corollary~\ref{cor:pos}. With suitable modifications in the proof of Theorem \ref{thm:main}, one can prove Theorem~\ref{gen}. Suppose $n,m$ are positive integers. Then it is known that $$(x^n-1,x^m-1)=x^{(n,m)}-1.$$ The greatest common divisor for one or two positive signs in the above is as follows:  

\begin{lemma}\label{lem:cyclo}
Suppose $n, m$ are positive integers. Then 
\begin{enumerate}[label=(\alph*)]
\item \label{lem:cyclo:a} $$(x^n+1, x^m+1)=\begin{cases}
x^{(n,m)}+1 &\mbox{ if $e(n)=e(m)$}\\
1 &\mbox{ otherwise.}
\end{cases}$$
\item \label{lem:cyclo:b} $$(x^n+1, x^m-1)=\begin{cases}
x^{(n,m/2)}+1 &\mbox{ if $e(m)\ge 2e(n)$}\\
1 &\mbox{ otherwise.}
\end{cases}$$
\end{enumerate}
\end{lemma} 

\begin{proof}
Let $n=2^{\alpha}n_1$ and $m=2^{\beta}m_1$ be  prime factorizations of $n$ and $m$ respectively. Then the proof can be carried out from the identities: $x^n+1=\prod\limits_{d|n_1}\Phi_{2^{\alpha+1}d}(x)$ and $x^m-1=\prod\limits_{i=0}^{\beta}\prod\limits_{d|m_1}\Phi_{2^id}(x).$ We omit the details.
\end{proof}

\begin{proof} [\unskip\nopunct]{\textbf{Proof of Theorem~\ref{thm:main}:}}
Suppose $f(x)=g(x)h(x)$ be a non-trivial factorization of $f(x).$
Since $|a_0|=|g(0)||h(0)|$ is a prime number, without loss of generality  we can assume $|h(0)|=|a_0|.$ Suppose $g(x)=f_1(x)f_2(x)\cdots f_l(x),$ where $f_i(x)$ is an irreducible factor of $g(x).$ Now we claim that for every $i,$ $f_i(x)$ is a cyclotomic polynomial. It is sufficient to prove that $f_1(x)$ is a cyclotomic polynomial.

It is clear that $f_1(0)=\pm 1.$ Let $\deg(f_1(x))=s$ and $z_1,z_2,\ldots,z_s$ be roots of $f_1(x)$. Then 
\begin{equation}\label{eq:prod}
\prod_{i=1}^s |z_i|=\frac{1}{|d|},
\end{equation} 
where $d$ is the leading coefficient of $f_1(x)$. 
If $|z_i|<1$ for some $i$, then 
\begin{equation*}
|a_0|=|-(a_{n_1}z_i^{n_1}+a_{n_2}z_i^{n_2}+\cdots+a_{n_r}z_i^{n_r})|< |a_{n_1}|+|a_{n_2}|+\cdots +|a_{n_r}|
\end{equation*} 
contradicts the hypothesis. So all of the roots of $f_1(x)$ lies in the region $|z|\ge 1$. But from Equation~(\ref{eq:prod}) we have $\prod_i |z_i|\le 1.$ Consequently $d=\pm 1.$ From Kronecker's theorem $f_1(x)$ is a cyclotomic polynomial. Next we will prove $g(x)=\gcd(x^{n_i}+\sgn(a_0a_{n_i})),$ where $1\le i\le r.$

 Let $\zeta$ be a primitive $t^{\text{th}}$ root of unity with $f(\zeta)=0$. Then 
\begin{equation}\label{eq2}
-a_0=a_{n_1}\zeta^{n_1}+a_{n_2}\zeta^{n_2}+\cdots+a_{n_r}\zeta^{n_r}.
\end{equation} 
Taking modulus on both sides
\begin{equation*}
|a_0|=|a_{n_1}\zeta^{n_1}+a_{n_2}\zeta^{n_2}+\cdots+a_{n_r}\zeta^{n_r}|=\sum\limits_{i=1}^r|a_{n_i}|.
\end{equation*}
From triangle inequality, the last two equations hold if and only if the ratio of any two part is a positive real number. Therefore, $a_{n_r}\zeta^{n_r-n_i}/a_{n_i}=|a_{n_r}\zeta^{n_r-n_i}/a_{n_i}|$ gives $\zeta^{n_r-n_i}=\sgn(a_{n_r}a_{n_i})$ for $1\le i\le r-1$. From \eqref{eq2}, we have
\begin{equation*}
-a_0= a_{n_i}\zeta^{n_i}\left[ \left|\frac{a_{n_1}}{a_{n_i}}\right|+\dots+\left|\frac{a_{n_{i-1}}}{a_{n_i}}\right|+1+\left|\frac{a_{n_{i+1}}}{a_{n_i}}\right|+\dots+\left|\frac{a_{n_r}}{a_{n_i}}\right|\right],
\end{equation*}
so that $\zeta^{n_i}=-\sgn(a_0a_{n_i})$. From $\zeta^{n_r-n_i}\zeta^{n_i}$, one gets the last equation. Remaining all the equations satisfied by $\zeta$ can be drawn from these $r$ equations. Conversely, if $\zeta$ satisfy each of the $r$ equations $x^{n_i}+\sgn(a_0a_{n_i})=0$ then $f(\zeta)=0$. It remains to show that each cyclotomic factor has multiplicity one. If not, suppose $\zeta$ satisfies the $r$ equations $x^{n_i}+\sgn(a_0a_{n_i})=0, 1\le i\le r$ and $f(\zeta)=0, f'(\zeta)=0$. Then the last relation gives 
\begin{equation*}
n_ra_{n_r}\zeta^{n_r-1}+\cdots+n_1a_{n_1}\zeta^{n_1-1}=0.
\end{equation*}
Using those $r$ equations satisfied by $\zeta$, we will end up with the relation  $n_r|a_{n_r}|+\cdots+|a_{n_1}|n_1=0$, which is not possible. 
\end{proof}

\begin{rem}\label{rem}
Let $g(x)$ be the  product of all cyclotomic factors of $f(x),$ where $f(x)$ is  as stated in Theorem~\ref{gen}. Then similar to Theorem~\ref{thm:main}, it can be seen that $g(x)$ is the greatest common divisor of the equations 
\begin{equation*}
x^{n_i}=-\sgn(a_0a_{n_i}), \quad 1\le i\le r.
\end{equation*}
\end{rem}

\begin{proof} [\unskip\nopunct]{\textbf{Proof of Corollary~\ref{cor:pos}:}} If  $g(x)=h(x^d)\in \Z[x]$, $d\ge 1$, then $g(x)$ has a cyclotomic factor if and only if $h(x)$  has a  cyclotomic factor. Thus it is sufficient to prove  the result for polynomials whose exponents are relatively prime. The result follows from Theorem \ref{thm:main} and Lemma \ref{lem:cyclo}.
\end{proof}

\section{Applications}\label{sec:app}
We apply the above theory to trinomials of the form $f(x)=ax^n+b\epsilon_1x^m+p\epsilon_2,$ where $a,b,p$ are positive integers, $a+b=p$, $p$ is a prime and $\epsilon_i\in \{-1, 1\}$. 

A. Schinzel\cite{AS} had found out the cyclotomic factors of trinomials. For the present family of polynomials, we will give the irreducibility criterion in a more simpler form. Similar studies related to trinomials can be found on \cite{sel},\cite{tver},\cite{AB},\cite{JS}. 

From Theorem \ref{thm:main}, $f(x)$ is separable over $\Q$. However, by using the discriminant formula for trinomials, one can prove the separability of $f(x)$ over $\Q$, where $|f(0)|$ is prime, in a  wider range. For example,  it is known that 
\begin{theorem}[C.R. Greenfield, D. Drucker~\cite{GD}]\label{th0}
The discriminant of the trinomial $x^n+ax^m+b$ is 
\begin{equation*}
D=(-1)^{\binom{n}{2}}b^{m-1}\left[ n^{n/d}b^{n-m/d}-(-1)^{n/d}(n-m)^{n-m/d}m^{m/d}a^{n/d}\right]^d
\end{equation*}
where $d=(n,m)$.
\end{theorem}

\begin{theorem}\label{sepa}
Let $a,b,p\in \N$, $b\le p$ and $p$ be a prime number. Then $f(x)=ax^n+b\epsilon_1x^m+p\epsilon_2,$ where $\epsilon_i\in\{-1,1\}$ is separable over $\Q$. 
\end{theorem}

\begin{proof}
 From Theorem \ref{th0}, the discriminant of $f$ is 
\begin{equation*}
D_f=(-1)^{\binom{n}{2}}(p\epsilon_2)^{n-m}a^{n-m-1}\left[ n^{n/d}(p\epsilon_2)^{n-m/d}a^{m/d}-(-1)^{n/d}(n-m)^{n-m/d}m^{m/d}(b\epsilon_1)^{n/d}\right]^d,
\end{equation*}
where $d=(n,m)$. Since the polynomial $x^v-\alpha, 0\ne\alpha\in \C,$ is separable over $\Q$, it is sufficient to consider $d=(n,m)=1$. $f(x)$ has multiple root if and only if $D_f=0$, i.e.,
\begin{equation*}
n^n(p\epsilon_2)^{n-m}a^m=(-1)^n(n-m)^{n-m}m^m(b\epsilon_1)^n.
\end{equation*}
Since $d=1$ and $n\ge 3$,  it is not possible that $b=1$. If $b>1$, $n|b$ and $p|m$ or $p|n-m$ give $m,n-m<n\le b\le p$, a contradiction. If $n=2$, then $f(x)=ax^2+b\epsilon_1x+p\epsilon_2$ has multiple roots if and only if $p|b$ and $p=4\epsilon_2a$, contradicting the fact that $p$ is a prime number.
\end{proof}
To determine the cyclotomic factors of trinomials, it can readily be observed that if $\epsilon_1=1, \epsilon_2=-1$ then $f(1)=0$. In fact $f(x)/x^{(n,m)}-1$ is an irreducible non-reciprocal  polynomial by Theorem~\ref{thm:main} and \ref{sepa}. We will consider the remaining cases below.
\begin{theorem} Let $f(x)=ax^n+b\epsilon_1x^m+p\epsilon_2$, where $a,b,p\in \N$, $p$ be a prime number, $\epsilon_i\in \{-1,1\}$ and $a+b=p$. Then $f(x)$ is irreducible except in the following cases:
\begin{multicols}{2}
\begin{enumerate}[label=(\alph*)]
\item\label{a}  $\epsilon_1=-1, \epsilon_2=1; e(n)<e(m)$, 
\item\label{b} $\epsilon_1=\epsilon_2=-1; e(n)>e(m)$,
\item\label{c}$\epsilon_1=\epsilon_2=1; e(n)=e(m)$, 
\end{enumerate} 
\end{multicols}
and in each of the above cases $f_c(x)=x^{(n,m)}+1$. 
\end{theorem}
\begin{proof}
Suppose $f(x)=ax^n-bx^m+p$ be reducible. From Theorem \ref{thm:main} the cyclotomic factors of $f(x)$ divide the equations $x^n+1=0$ and $ x^m-1=0$. From Lemma \ref{lem:cyclo}, we have, 
\begin{equation*}
f_c(x)|(x^n+1,x^m-1)=\begin{cases} 
x^{(n,m/2)}+1 & \mbox{ if $e(m)\ge 2e(n)$}\\
1 & \mbox{ otherwise}
\end{cases}
\end{equation*}
Let $n=2^{\alpha}n_1, m=2^{\beta}m_1$ be the prime factorization of $n$ and $m$ respectively, and $d=(n,m/2)=2^{\min(\alpha,\beta)}(n_1,m_1)$. If $\alpha\ge \beta,$ then $f(x)=f_1(x^{2^{\beta}})$ with $f_1(x)$ irreducible by Corollary \ref{cor:pos}. Therefore, $f(x)$ is reducible if and only if $\alpha<\beta$, that is, $e(n)<e(m)$. Also, if $e(n)<e(m)$ then $(n,m/2)=(n,m)$.  Conversely, let $\zeta$ be a root of $x^d+1$ where $d=(n,m)$. Then
\begin{equation*}
f(\zeta)=a\zeta^n-b\zeta^m+p=-a-b+p=0.
\end{equation*} 
From Theorem \ref{thm:main} or Theorem \ref{sepa}, we get $f_c(x)=x^{(n,m)}+1$. In a similar fashion, part \ref{b}, \ref{c} can be proved. We omit the details.
\end{proof}
 
W. Ljunggren\cite{lju} has proved that the quadrinomials of the form $g(x)=x^n+\epsilon_1x^m+\epsilon_2x^r+\epsilon_3$ can have at most one irreducible non-reciprocal factor apart from its cyclotomic factors. However, Mills \cite{mills} proved that Ljunggren's result is true except for four families in which corresponding quadrinomials factor as a  product of two non-reciprocal irreducible polynomials apart from its cyclotomic factors. For example,
$$x^{8k}-x^{7k}-x^k-1=(x^{2k}+1)(x^{3k}-x^k-1)(x^{3k}-x^{2k}-1), \text{ for all } k\in \N.$$
Therefore, in either case, determination of  the separability of quadrinomials is same as to check separability of its cyclotomic parts. There are quadrinomials which are not separable. For example, 
$x^4+x^3+x+1=(x+1)^2(x^2-x+1).$

\begin{theorem}
Suppose $f(x)=x^n+\epsilon_1x^m+\epsilon_2x^r+\epsilon_3$ be a polynomial  of degree $n$, $\epsilon_i\in \{\, 1,-1\,\}$. Then $f(x)$ is separable if  $f(\pm 1)\ne 0$. 
\end{theorem}

\begin{proof}
Without loss of generality, we can take $n\le m+r$. Otherwise we replace $f(x)$ by $\epsilon_3x^nf(x^{-1})$. Let $h(x)=(f(x),f'(x))$. Then $h(x)|nf(x)-xf'(x)=(n-m)\epsilon_1 x^m+(n-r)\epsilon_2x^r+n\epsilon_3$. If $n<m+r$, then all the roots of $nf(x)-xf'(x)$ lies outside the unit circle. For if $|z|\le 1$ be a root of $nf(x)-xf'(x)$, then 
\begin{equation*}
n=|(n-m)\epsilon_1z^m+(n-r)\epsilon_2z^r|\le 2n-m-r
\end{equation*}
contradicts the fact that $n<m+r$. Since all the roots of $nf(x)-xf'(x)$ lies in the region $|z|>1$, we have $|h(0)|>1$. This contradicts the fact that $h(0)|f(0)$.

Let $n=m+r$. From \cite{lju} and \cite{mills}, $h(x)$ is a cyclotomic polynomial. Therefore, it is sufficient to consider $(m,r)=1$. Then $nf(x)-xf'(x)$ satisfies Theorem~\ref{gen} and  by Remark~\ref{rem}, if $h(x)\ne 1$ then $h(x)$ divides either $x+1$ or $x-1$. 
\end{proof}

\end{document}